\documentclass[letterpaper, 10 pt, conference]{ieeeconf}  

\IEEEoverridecommandlockouts                              

\usepackage{cite}
\usepackage{amsmath,amssymb,amsfonts}
\usepackage{graphicx}
\usepackage{textcomp}
\usepackage{xcolor}
\usepackage{hyperref}
\usepackage{amsbsy}
\usepackage{algorithmic}
\usepackage[]{algorithm}
\usepackage{enumerate}
\usepackage{caption}
\usepackage{subcaption}
\usepackage{multirow}
\usepackage{tabu}
\usepackage[normalem]{ulem}
\usepackage{array}
\newcolumntype{P}[1]{>{\centering\arraybackslash}m{#1}}

\DeclareMathOperator*{\argmin}{arg\,min} 
\providecommand{\ip}[2]{\langle #1, #2 \rangle} 

\def\la{\langle}
\def\ra{\rangle}

\def\1{{\bf 1}}

\def\g{\gamma}

\def\R{\mathbb{R}}

\def\argmin{\mathop{\rm argmin}}
\def\be{\begin{equation}}
\def\ee{\end{equation}}

\def\bit{\begin{itemize}}
\def\eit{\end{itemize}}

\newcommand{\remove}[1]{}

\providecommand{\ip}[2]{\langle #1, #2 \rangle} 

\newtheorem{assumption}{Assumption}

\newtheorem{theorem}{\bf Theorem}

\newtheorem{lemma}{\bf Lemma}

\title{\LARGE \bf
Popov Mirror-Prox for solving Variational Inequalities
}

\author{Abhishek Chakraborty$^{1}$ and Angelia Nedi\'c$^{1}$
\thanks{$^{1}$Abhishek Chakraborty and Angelia Nedi\'c are with the Department of Electrical, Computer, and Energy Engineering, Arizona State University, Tempe, USA,
        Email: {\tt\small achakr61@asu.edu}, {\tt\small Angelia.Nedich@asu.edu}. This work has been supported by the ONR award N00014-21-1-2242 and the NSF grant CIF-2134256.}}%

\begin{document}

\maketitle
\thispagestyle{empty}
\pagestyle{empty}

\begin{abstract}
We consider the mirror-prox algorithm for solving monotone Variational Inequality (VI) problems. As the mirror-prox algorithm is not practically implementable, except in special instances of VIs (such as affine VIs), we consider its implementation with Popov method updates. We provide convergence rate analysis of our proposed method for a monotone VI with a Lipschitz continuous mapping. We establish a convergence rate of $O(1/t)$, in terms of the number $t$ of iterations, for the dual gap function. Simulations on a two player matrix game corroborate our findings.
\end{abstract}

\section{Introduction}

Variational Inequalities (VIs) are important for the study of equilibrium problems arising in economics \cite{jofre2007variational}, game theory \cite{hu2011variational}, multi-agent reinforcement learning \cite{lanctot2017unified}, etc. To solve VI problems,  conventional projection-based  methods have been developed in~\cite{facchinei2003finite, korpelevich1976extragradient, popov1980modification, tseng1995linear, malitsky2015projected}. 
Korpelevich \cite{korpelevich1976extragradient} and Popov \cite{popov1980modification} methods have been mostly used to solve VI problems with monotone mappings, as the projection method can diverge for such mappings~\cite{mokhtari2020unified}. 

The mirror descent algorithm was introduced in~\cite{nemirovskij1983problem} in the context of optimization problems, 
and has dimension (of the variable) independent rates for non-smooth functions \cite{nesterov2005smooth, nemirovski2004prox}.
It has been also shown to have a favorable scaling with the gradient size~\cite{bubeck2015convex} in comparison with other projection-based first-order methods.
The mirror-prox algorithm in the context of monotone VIs has been proposed in~\cite{nemirovski2004prox} for the deterministic VIs and in~\cite{juditsky2011solving} for the stochastic VIs. However, a direct implementation of the mirror-prox method is rarely possible. 
To address this issue, the  work in~\cite{nemirovski2004prox} has proposed the use of Korpelevich method (aka extra-gradient method) to develop an implementable version of the mirror-prox method. The Korpelevich mirror-prox algorithm requires two mapping evaluations per iteration. Instead, a single mapping evaluation per iteration can be used if the Popov method is employed in mirror-prox as proposed in~\cite{semenov2017version} for pseudo-monotone VIs. 
However, the convergence rate of the algorithm has remained an open question. On the other hand, the work in~\cite{jiang2022generalized} presents the optimistic/Popov methods for saddle point problems and provide convergence rates. The work in~\cite{azizian2023rate} studies similar algorithms but with strongly monotone mapping. None of the aforementioned work provides convergence rate of the Popov mirror-prox method for monotone VIs, which we address in this paper. 

\noindent \textbf{Motivations and Contributions:} Algorithms using the mirror mapping as KL divergence has become popular in the context of machine learning with probabilistic loss \cite{yang2019learning, chakraborty2022sparse}, multi-agent reinforcement learning \cite{zhong2024heterogeneous}, \cite{lan2023policy}, etc. 
The mirror-prox map preserves the positivity of the iterates (such as the use of the KL divergence), which allows for easy projections on the set using scaling instead of the costly iterative simplex projection algorithms \cite{nemirovski2004prox, blondel2014large}. 
Moreover, the mirror-prox maps data from the actual space to a dual space and, hence, the problem geometry can be effectively used with a suitable selection of mapping functions~\cite{bubeck2015convex}. In addition, recent paper~\cite{anagnostides2023optimistic} representing Markov Games as VIs provides a potential application of the mirror-prox algorithms for efficient computations. 
Motivated by these applications, {\it we investigate the Popov mirror-prox algorithm in the context of monotone VIs defined in a general Euclidean space not necessarily finite dimensional. We  provide its convergence rate which has not been done in the existing literature even for VIs in a finite dimensional Euclidean space}. 
We obtain $O(1/t)$ convergence rate, with the number $t$ of iterations, for the dual gap function. The rate is of  the similar order as that of the Korpelevich mirror-prox~\cite{nemirovski2004prox}.

\noindent \textbf{Flow of the paper:} Section \ref{sec:intro} introduces the VI problem, related assumptions and the mirror-prox algorithm. Section~\ref{sec_algo} provides the Popov mirror-prox algorithm, while its convergence analysis is given in Section~\ref{sec_popov_rates}. Section~\ref{sec_simulation} provides simulation results of the proposed method for a two player matrix game. 
Finally, some discussion and future directions are given in Section~\ref{sec_conclusion}.

\noindent \textbf{Notations:} Throughout the paper, we consider the Euclidean space $E$ equipped with an inner product $\la\cdot,\cdot\ra$ and a norm $\|\cdot\|$, which need not necessarily be the one induced by the
inner product. We let $E^*$ be the dual space of $E$ equipped with the dual (conjugate) norm $\|\cdot\|_*$ defined by 
\[\|y\|_*=\max_{\|x\|\le 1}\la y,x\ra.\] 


\section{Problem, Assumptions and Preliminaries}\label{sec:intro}
In this section, we formulate the VI problem of our interest,
present the conceptual mirror-prox method as proposed in~\cite{nemirovski2004prox}, and introduce some definitions related to the mirror-prox algorithm.

\subsection{VI Problem and Assumptions} 
Given a set $X \subseteq E$ and a mapping $F: X \rightarrow E^*$, we consider the problem of finding $x^* \in X$ such that
\begin{align}
    \la F(x^*), x-x^* \ra \geq 0 , \qquad \forall x \in X. \label{VI_problem}
\end{align} 
The preceding problem is a Variational Inequality problem denoted by VI$(X, F)$. We take the following assumptions on the set $X$ and the mapping $F(\cdot)$ for the VI$(X, F)$.
\begin{assumption}\label{asum_compact}
    The set $X$ is convex and compact.
\end{assumption}
%
%
\begin{assumption}\label{asum_monotone}
    The mapping $F : X \rightarrow E^*$ is monotone over the set $X$, i.e.,
    \begin{align}
        \la F(x) - F(y), x-y \ra \geq 0 , \qquad \forall x,y \in X .
    \end{align}
\end{assumption}
\begin{assumption}\label{asum_lipschitz}
    The mapping $F : X \rightarrow E^*$ is Lipschitz continuous over the set $X$ with a constant $L>0$, i.e.,
    \begin{align}
        \| F(x) - F(y) \|_* \leq L \|x-y\| , \qquad \forall x,y \in X .
    \end{align}
\end{assumption}
A {\it strong solution} to VI$(X,F)$ is a point $x^*\in X$ such that inequality~\eqref{VI_problem} holds. Another concept of a solution to the VI$(X,F)$ exists, namely a weak solution. 
A {\it weak solution} to VI$(X,F)$ is a point $x^*\in X$ such that 
\begin{align}
    \la F(x), x-x^* \ra \geq 0 , \qquad \forall x \in X. \label{weak-sol}
\end{align} 
When the set $X$ is convex and compact, and the mapping $F(\cdot)$ is monotone on $X$ (cf.\ Assumptions~\ref{asum_compact}--\ref{asum_monotone}) a weak solution always exists~\cite{JuditskyNemirovsky2016}.
Moreover, when the mapping $F(\cdot)$ is
monotone and continuous on $X$, a weak solution is also
a strong solution~\cite{JuditskyNemirovsky2016}.
Under Assumption~\ref{asum_lipschitz}, the mapping $F(\cdot)$ is continuous. Hence, under Assumptions~\ref{asum_compact}--\ref{asum_lipschitz}, the weak and the strong solutions of the VI$(X,F)$ exist and coincide.

An alternative way to characterize a weak solution $x^*$ for the VI problem~\ref{VI_problem} is through the use of a dual gap function, denoted by $G(x)$  and defined by
\begin{align}
    G(x) = \max_{u \in X} \ip{F(u)}{x - u}, \qquad\forall x \in X.\label{dg_func}
\end{align}
Note that $G(x) \geq 0$ for any $x \in X$, and $G(x^*)=0$ if and only if $x^*\in X$ is a weak solution to the VI$(X,F)$~\cite{JuditskyNemirovsky2016} (see~\cite{zhang2003dual} for a finite dimensional Euclidean space). Therefore, under our
Assumptions~\ref{asum_compact}--\ref{asum_lipschitz}, we have that $G(x^*)=0$ if and only if $x^*\in X$ is a strong solution to the VI$(X,F)$. Thus, the quantity $G(x)$ can be viewed as a measure of the quality of an approximate solution $x$ to the VI$(X,F)$.



\subsection{Preliminaries and Mirror-prox Algorithm}
The mirror-prox algorithm uses a distance function induced by a differentiable and strongly convex function $\psi: X \rightarrow \R$, 
satisfying the following assumption.
\begin{assumption}\label{asum_phi}
    The function $\psi: X \rightarrow \R$ is continuously differentiable and strongly convex with the constant $\alpha>0$, i.e., for all $x,y \in X$,
    \begin{align}
        \psi(y) \geq \psi(x) + \la \nabla \psi(x), y-x \ra + \frac{\alpha}{2} \| y - x \|^2 . \nonumber
    \end{align}
\end{assumption}

The Fenchel dual of the function $\psi(\cdot)$, 
restricted to the space $E^*$, is given by
\begin{align}
    \Omega(\xi) = \max_{x \in X} [\ip{\xi}{x} - \psi(x)] , \qquad \forall \xi \in E^*. \label{Frenchel_dual_def}
\end{align}
Furthermore, for any $u \in X$, we define the functions $\Omega_u: E^* \rightarrow \R$ and $H_u: X \rightarrow \R$, as follows:
\begin{align}
    &\Omega_u(\xi) := \Omega(\xi) - \ip{\xi}{u} , \qquad \forall \xi \in E^*, \label{def_Omega_u}\\
    &H_u(x) := \Omega_u(\nabla \psi(x)) , \qquad \forall x \in X. \label{def_H_u}
\end{align}
These functions will be important in the subsequent analysis of the Popov mirror-prox method. The function $H_u(x)$ captures a form of distance between the points $u \in X$ and $x \in X$, and it is an artifact of the analysis. Moreover, we will use the prox-mapping defined as follows:
For every $x \in X$ and $\xi \in E^*$, the prox-mapping $P_x: E^* \rightarrow X$  is given by
\begin{align}
    P_x(\xi) = \argmin_{y \in X} [\psi(y) + \ip{y}{\xi - \nabla \psi(x)}] . \label{prox_map}
\end{align}
 We next present a lemma from \cite{nemirovski2004prox} that shows the Lipschitz continuity of the prox-mapping and a relation for the difference between $H_u(\cdot)$ evaluated at two different points.
\begin{lemma}[Lemma 2.1 in \cite{nemirovski2004prox}] \label{lem_prelim}
    Under Assumptions \ref{asum_compact} and \ref{asum_phi}, for all $x,u \in X$ and $\xi , \eta \in E^*$, the following relations hold:
    \begin{align}
        &\|P_x(\xi) - P_x(\eta)\| \leq \frac{1}{\alpha} \|\xi - \eta \|_*, \nonumber\\
        &H_u(P_x(\xi)) - H_u(x) \leq \ip{\xi}{u-P_x(\xi)} \nonumber\\
        & \hspace{2cm}+ [\psi(x) + \ip{\nabla \psi(x)}{P_x(\xi) - x} - \psi(P_x(\xi))], \nonumber
    \end{align}
    where $\alpha>0$ is the strong convexity constant of the function $\psi(\cdot)$.
\end{lemma}

The first relation of Lemma~\ref{lem_prelim} just states that the prox-mapping $P_x$, defined in \eqref{prox_map}, is Lipschitz continuous with the constant $1/\alpha$. The second relation of Lemma~\ref{lem_prelim} is a key result that is used in the analysis of the mirror-prox algorithm in~\cite{nemirovski2004prox}, which will also be crucial in our analysis later on.

We next provide the conceptual mirror-prox method  of~\cite{nemirovski2004prox}.
%
%
\begin{algorithm}
		\caption {Conceptual mirror-prox \cite{nemirovski2004prox}}
		\label{algo_mirror_prox}
		\begin{algorithmic}[1]
                \STATE \textbf{Initialize:} $x_0 \in X$
			\FOR{$t=1,2,\ldots$}
                \IF{$P_{x_{t-1}}(F(x_{t-1})) = x_{t-1}$}
                \RETURN $x_{t-1}$ as the solution
                \ELSE
                \STATE \textbf{Choose} $\gamma_t > 0$ and $y_t \in X$ to satisfy the relation 
                \begin{align}
                    &\ip{y_t - P_{x_{t-1}}(\gamma_t F(y_t))}{\gamma_t F(y_t)} \nonumber\\
                    &\hspace{-2mm} + [\psi(x_{t-1}) + \ip{\nabla \psi(x_{t-1})}{P_{x_{t-1}}(\gamma_t F(y_t))-x_{t-1}} \nonumber\\
                    &\hspace{1cm}- \psi(P_{x_{t-1}}(\gamma_t F(y_t)))] \leq 0 \label{mirror_prox_ineq}
                \end{align}
                \vspace{-4mm}
                \STATE \textbf{Update:  } $x_t = P_{x_{t-1}}(\gamma_t F(y_t))$             
                \ENDIF
                \STATE \textbf{Output: } $y^{(t)} = \frac{\sum_{\tau=1}^t \gamma_{\tau} y_{\tau}}{\sum_{\tau=1}^t \gamma_{\tau}}$
			\ENDFOR
		\end{algorithmic}
\end{algorithm}	
%

%

The output of Algorithm \ref{algo_mirror_prox} is the weighted combination of the iterates $y_t$ as given in step 9 of the algorithm. 
Note that step 6 of Algorithm~\ref{algo_mirror_prox} does not 
specify an update equation for obtaining $y_t \in X$ satisfying \ref{mirror_prox_ineq}.
Thus, the mirror-prox method
can be computationally prohibitive to implement. 
Setting $y_t = P_{x_{t-1}}(\gamma_t F(y_t))$ and using the convexity of $\psi(\cdot)$ ensures that \eqref{mirror_prox_ineq} holds. However, computing $y_t$ in such a way is costly, as we need to solve a VI with a monotone map $F_{t-1}(y_t) = \nabla \psi(y_t) - \nabla \psi(x_{t-1}) + \gamma_t F(y_t)$ (see \cite{nemirovski2004prox}). Hence $x_t = P_{x_{t-1}}(\gamma_t F(y_t))$ is impossible to implement.

Note that the definition of $x_t$ in step 7 of Algorithm~\ref{algo_mirror_prox} and the definition of the prox-mapping $P_{x_{t-1}}(\cdot)$ in~\eqref{prox_map} imply that 
\begin{align}\label{eq-xt}
    x_t= \argmin_{y \in X} [\ip{\gamma_t F(y_t) - \nabla \psi(x_{t-1})}{y} +\psi(y)].
\end{align}
Thus, determining the point $x_t$ in a closed form is not always easy, unless the set $X$ has a favorable structure so that it is easy to project on, such as, for example box constraints, norm ball, etc. In this paper, we {\it assume that the set $X$ has a simple structure for the projection}, and we focus only on simplifying the update for $y_t$ so that inequality~\eqref{mirror_prox_ineq} is satisfied. 

As the step of finding $y_t$ satisfying~\eqref{mirror_prox_ineq} is impossible to implement, multiple update steps to reach $x_t$  were proposed in \cite{nemirovski2004prox}. The main insight for these updates comes from the fact that, under Assumption~\ref{asum_lipschitz}, the prox-mapping $P_x(\cdot)$ is contractive for suitably chosen points with a constant step size $\gamma_t=\gamma$. In particular, by using Lemma~\ref{lem_prelim} with $\xi = \gamma F(y_1)$ and $\eta = \gamma F(y_2)$, for some $y_1,y_2 \in X$, 
we have that 
\begin{align}
    \|P_x(\gamma F(y_1)) - P_x(\gamma F(y_2))\| &\leq \frac{1}{\alpha} \|\gamma F(y_1) - \gamma F(y_2) \|_* \nonumber\\
    &\leq \frac{\gamma L}{\alpha} \|y_1 - y_2 \|,  \label{eq_contraction}
\end{align}
where $L$ is the Lipschitz constant of the mapping $F(\cdot)$ from~Assumption~\ref{asum_lipschitz}. With a choice of the step size $\gamma$, such that $\gamma L/\alpha <1$, from~\eqref{eq_contraction} we see that
the prox-mapping is contractive (at points of the form $\gamma F(y)$) and, thus, the solution $x_t$ in~\eqref{eq-xt} can be obtained geometrically fast. It turned out that $x_t$ 
can be obtained with only two update steps~\cite{nemirovski2004prox}. We consider such updates in the next section for the mirror-prox method combined with the Popov method.



\section{Popov Mirror-prox Method}\label{sec_algo}


In this section, we present the Popov mirror-prox method. We start with the two step update of \cite{nemirovski2004prox} with a starting iterate $x \in X$, and two points $\xi, \eta \in E^*$:
\begin{align}
    &\widetilde y = \argmin_{z \in X} [\ip{\gamma \xi - \nabla \psi(x)}{z} + \psi(z)] , \nonumber\\
    &\widetilde x = \argmin_{z \in X} [\ip{\gamma \eta - \nabla \psi(x)}{z} + \psi(z)] , \label{two_step_update}
\end{align}
where $\g>0$ is a step size.
Based on the selection of the constant step size $\gamma>0$ and the points $\xi$ and $\eta$, the relation in \eqref{two_step_update} can be used as an implementable version of the conceptual mirror-prox algorithm, i.e., Algorithm \ref{algo_mirror_prox}. The paper \cite{nemirovski2004prox} studies the Korpelevich mirror-prox by selecting the points $x=x_t$, $\widetilde y = y_{t+1}$, $\widetilde x= x_{t+1}$, $\xi = F(x_t)$ and $\eta=F(y_{t+1})$, for iterate index $t = 1, \ldots$, and hence requires the mapping computation for two sequences, $\{x_t\}$ and $\{y_t\}$, which can be costly. Instead, the mapping computation can be done for a single sequence and an old mapping evaluation can be reused, which is the basis of our Popov mirror-prox method presented in Algorithm~\ref{algo_Popov}. 
%
%
%
%

\begin{algorithm}
		\caption {Popov mirror-prox method}
		\label{algo_Popov}
		\begin{algorithmic}[1]
			\REQUIRE{step size $\gamma>0$}
                \STATE \textbf{Initialize:} $x_0,y_0 \in X$
			\FOR{$t=1,\ldots$}
                \STATE \textbf{Update}
                \begin{align*}
                    &y_{t+1} = \argmin_{z \in X} [\ip{\gamma F(y_{t}) - \nabla \psi(x_{t})}{z} + \psi(z)] \\
                    &x_{t+1} = \argmin_{z \in X} [\ip{\gamma F(y_{t+1}) - \nabla \psi(x_{t})}{z} + \psi(z)]
                \end{align*}
                \vspace{-2mm}
                \STATE \textbf{Output:  } $y^{(t+1)} = \frac{1}{t+1}
                    \sum_{\tau=0}^t y_{\tau+1}$
			\ENDFOR
		\end{algorithmic}
\end{algorithm}	
%
%

In Algorithm \ref{algo_Popov}, the updates for $y_{t+1}$ and $x_{t+1}$ are obtained via relations in~\eqref{two_step_update},
by letting $x=x_t$, $\widetilde y = y_{t+1}$, $\widetilde x= x_{t+1}$, $\xi = F(y_t)$, $\eta=F(y_{t+1})$, and using a fixed step size $\gamma>0$. Since both $y_{t+1}$ and $x_{t+1}$ at step 3 of Algorithm 2 are solutions to strongly convex problems, they are uniquely defined. 
Moreover,
by using the optimality conditions, we can see that the iterates $y_{t+1}$ and $x_{t+1}$ are such that for any $z \in X$, the following inequalities hold:
\begin{align}
    &\ip{\gamma F(y_t) - \nabla \psi(x_t) + \nabla \psi(y_{t+1})}{z-y_{t+1}} \geq 0 , \label{opt_y_pop}\\
    &\ip{\gamma F(y_{t+1}) - \nabla \psi(x_t) + \nabla \psi(x_{t+1})}{z-x_{t+1}} \geq 0 \label{opt_x_pop}.
\end{align}
The first step of the update produces an iterate sequence $\{y_{t}\}$, while the next step produces another iterate sequence $\{x_{t}\}$. The mapping $F(\cdot)$ is evaluated with respect to the sequence $\{y_t\}$ and the old mapping is reused resulting in a single mapping computation per loop of the algorithm. The output $y^{(t+1)}$ of Algorithm 2 is  the average of the iterates $y_k$, $1\le k\le t+1$.



\section{Convergence Rate analysis}\label{sec_popov_rates}
Before analyzing the convergence rates of Algorithm \ref{algo_Popov}, we first show a generic lemma that establishes the relationship between the different points used in describing the two step updates presented in \eqref{two_step_update}.
\begin{lemma}[Lemma 3.1 in \cite{nemirovski2004prox}]\label{lem_two_step}
    Under Assumptions \ref{asum_compact} and \ref{asum_phi}, for the updates $\tilde y$ and $\tilde x$ in \eqref{two_step_update}, the following relations hold:

    \noindent (i) The difference between the two iterates $\widetilde y \in X$ and $\widetilde x \in X$ satisfies
    \begin{align}
        \|\widetilde y-\widetilde x\| \leq \frac{\gamma}{\alpha} \|\xi-\eta\|_*, 
    \end{align}
    where $\alpha>0$ is the strong convexity constant of the function $\psi(\cdot)$.\\
    \noindent (ii) The inner product $\ip{\gamma\eta}{\widetilde y-u}$ can be upper bounded as follows 
    \begin{align}
        \ip{\gamma\eta}{\widetilde y-u} \leq H_u(x) - H_u(\widetilde x) + \delta , \label{eq_dual_gap}
    \end{align}
    where $\delta = \ip{\gamma\eta}{\widetilde y-\widetilde x} + [\psi(x)+\ip{\nabla \psi(x)}{\widetilde x -x} - \psi(\widetilde x)]$.

    \noindent (iii) Moreover, the term $\delta$ in~\eqref{eq_dual_gap} can be upper bounded by $\epsilon$ where
    \begin{align}
        \epsilon &= \ip{\gamma(\eta-\xi)}{\widetilde y-\widetilde x} + [\psi(x)+\ip{\nabla \psi(x)}{\widetilde y -x}\nonumber\\
        &\hspace{2cm}+\ip{\nabla \psi(\widetilde y)}{\widetilde x -\widetilde y} - \psi(\widetilde x)] . \label{eq_eps}
    \end{align}

    \noindent (iv) In addition,  $\epsilon$ can be upper bounded as follows
    \begin{align}
        \epsilon \leq \alpha^{-1} \gamma^2 \|\xi-\eta\|^2_* - \frac{\alpha}{2} [\|\widetilde y-x\|^2 + \|\widetilde y-\widetilde x\|^2] . \label{eps_bound}
    \end{align}
\end{lemma}
%

%
Lemma \ref{lem_two_step} will be used to establish the convergence rate of Algorithm \ref{algo_Popov}. 
Note that the first two equations of Lemma \ref{lem_two_step} have a lot of resemblance with \eqref{eq_contraction} and the relations presented in Lemma~\ref{lem_prelim}. The relation in~\eqref{eq_dual_gap} is useful for analyzing the convergence rate of Algorithm~\ref{algo_Popov}. 

From now onwards, since all the quantities will be iterate index dependent, we will use Lemma \ref{lem_two_step} with $\delta_t$ in place of $\delta$, and $\epsilon_t$ in place of $\epsilon$, where $t$ is the iterate index. 

Next, we present a lemma that establishes an upper bound for the sum of $\epsilon_t$ for $t$ ranging from $1$ to some $T\ge 1$.

\begin{lemma}\label{lem_eps_bound}
    Let Assumptions \ref{asum_compact}, \ref{asum_lipschitz}, and \ref{asum_phi} hold. Consider Algorithm~\ref{algo_Popov} with a constant step size $\gamma$ satisfying $0< \gamma \leq \frac{\alpha}{2L}$. Then, we have for all $T\ge 1$,
    \begin{align}
        \sum_{t=0}^{T} \epsilon_t \leq \frac{2\gamma^2 L^2}{\alpha} \|y_0 - x_0\|^2 . \nonumber
    \end{align}
\end{lemma}

\begin{proof}
    We use relation \eqref{eps_bound} in Lemma \ref{lem_two_step} with $x=x_t$, $\widetilde y = y_{t+1}$, $\widetilde x= x_{t+1}$, $\xi = F(y_t)$, $\eta=F(y_{t+1})$, and $\epsilon = \epsilon_t$. Thus, we obtain
    \begin{align}
        &\epsilon_t \leq \frac{\gamma^2}{\alpha} \|F(y_t) - F(y_{t+1})\|^2_* \nonumber\\
        & - \frac{\alpha}{2} [\|y_{t+1}-x_t\|^2 + \|y_{t+1}-x_{t+1}\|^2] . \label{lem3_bound1}
    \end{align}
    Next, we concentrate on the first term on the right hand side of \eqref{lem3_bound1}. We add and subtract $F(x_t)$ to that term, use the triangle inequality, and the inequality $(a+b)^2 \leq 2 (a^2 + b^2)$, to obtain
    \begin{align}
        &\|F(y_t) - F(y_{t+1})\|^2_* \nonumber\\
        &= \|(F(y_t) - F(x_t)) - (F(y_{t+1})-F(x_t))\|^2_* \nonumber\\
        &\leq 2 \|F(y_t) - F(x_t)\|^2_* + 2 \|F(y_{t+1})-F(x_t)\|^2_* . \nonumber
    \end{align}
    Using the Lipschitz continuity of the mapping $F(\cdot)$ (Assumption~\ref{asum_lipschitz}), from the preceding inequality we further obtain
    \begin{align}
        \|F(y_t) - F(y_{t+1})\|^2_* \leq &2L^2 \|y_t - x_t\|^2 \nonumber\\
        &+ 2L^2 \|y_{t+1}-x_t\|^2_* . \label{lem3_bound2}
    \end{align}
    Now, substituting the estimate in \eqref{lem3_bound2} back in relation~\eqref{lem3_bound1}, we find that
    \begin{align}
        \epsilon_t &\leq \left(\frac{2\gamma^2 L^2}{\alpha} - \frac{\alpha}{2} \right) \|y_{t+1}-x_t\|^2 - \frac{\alpha}{2} \|y_{t+1}-x_{t+1}\|^2 \nonumber\\
        & \hspace{1cm}+ \frac{2\gamma^2 L^2}{\alpha} \|y_t - x_t\|^2 . \nonumber
    \end{align}
    Next, we sum $\epsilon_t$ for $t=0,1,\ldots,T$, and obtain
    \begin{align}
        \sum_{t=0}^{T} \epsilon_t &\leq \frac{2\gamma^2 L^2}{\alpha} \|y_0 - x_0\|^2 - \frac{\alpha}{2} \|y_{T+1}-x_{T+1}\|^2 \nonumber\\
    &- \sum_{t=0}^{T-1} \left( \frac{\alpha}{2} - \frac{2\gamma^2 L^2}{\alpha} \right) \|y_{t+1} - x_{t+1}\|^2 \nonumber\\
    &- \sum_{t=0}^{T} \left( \frac{\alpha}{2} - \frac{2\gamma^2 L^2}{\alpha} \right) \|y_{t+1} - x_{t}\|^2 . \label{lem3_bound3}
    \end{align}
    The step size $\gamma$ is selected to make the third and fourth terms on the right hand side of \eqref{lem3_bound3} non-positive, which holds when  $0 < \gamma \leq \frac{\alpha}{2L}$. With this step size selection, we can drop the non-positive terms on the right hand side of \eqref{lem3_bound3} and, thus, arrive at the desired relation.
\end{proof}

We now provide the convergence rate of Algorithm \ref{algo_Popov} in terms of the dual gap function $G(\cdot)$ defined in \eqref{dg_func}. The proof uses Lemma~\ref{lem_two_step} and Lemma \ref{lem_eps_bound}.

\begin{theorem}\label{thm_popov_rates}
Let Assumptions \ref{asum_compact}--\ref{asum_phi} hold.
Consider Algorithm \ref{algo_Popov} with a constant step size $\gamma = \frac{\alpha}{2L}$. Then, for the dual gap function $G(\cdot)$, evaluated at the average $y^{(T+1)}$ of the iterates, we have for all $T\ge 1$,
    \begin{align*}
        G(y^{(T+1)}) \leq \frac{2L}{(T+1) \alpha}  \max_{u \in X} B_{\psi}(u,x_0) + \frac{L}{T+1} \|y_0 - x_0\|^2 ,
    \end{align*}
    where $y^{(T+1)} = \frac{1}{T+1} \sum_{t=0}^T y_{t+1}$ and $B_{\psi}(\cdot,\cdot)$ is the Bregman divergence induced by $\psi(\cdot)$, 
    defined by $B_{\psi}(u,x) = \psi(u) - \psi(x) - \ip{\nabla \psi(x)}{u-x}$ for all $u,x\in X$.
\end{theorem}
\begin{proof}
    We use~\eqref{eq_dual_gap} of Lemma \ref{lem_two_step} with the following identifications $x=x_t$, $\widetilde y = y_{t+1}$, $\widetilde x= x_{t+1}$, $\xi = F(y_t)$, $\eta=F(y_{t+1})$, $\delta=\delta_t$, and $\epsilon = \epsilon_t$. We thus obtain
    \begin{align}
         \ip{\gamma F(y_{t+1})}{y_{t+1} - u} \leq H_u(x_t) - H_u(x_{t+1}) + \delta_t . \label{thm_eq1}
    \end{align}
    By Lemma \ref{lem_two_step}(iii), with $\delta=\delta_t$ and $\epsilon=\epsilon_t$, we see that $\delta_t \leq \epsilon_t$.
    Hence, \eqref{thm_eq1} can be further upper bounded as follows
    \begin{align}
        \ip{\gamma F(y_{t+1})}{y_{t+1} - u} \leq H_u(x_t) - H_u(x_{t+1}) + \epsilon_t . \nonumber
    \end{align}
    By summing the preceding relation for $t=0,1, \ldots, T$, for any $T\ge 1$, we obtain 
    \begin{align}
        \sum_{t=0}^T \ip{\gamma F(y_{t+1})}{y_{t+1} - u} \leq H_u(x_0) - H_u(x_{T+1}) + \sum_{t=0}^T \epsilon_t . \label{thm_eq2}
    \end{align}
    Using the monotonicity of the mapping $F(\cdot)$ (Assumption~\ref{asum_monotone}), we can lower bound the term on the left hand side of~\eqref{thm_eq2} 
    as follows:
    \begin{align*}
        &\sum_{t=0}^T \ip{\gamma F(y_{t+1})}{y_{t+1} - u} \geq \sum_{t=0}^T \ip{\gamma F(u)}{y_{t+1} - u} 
        \end{align*}
Since $\sum_{t=0}^T y_{t+1}=(T+1) y^{(T+1)}$,  it follows that 
        \begin{align}
        \sum_{t=0}^T \ip{\gamma F(y_{t+1})}{y_{t+1} - u} \ge 
         \gamma (T+1) \la F(u) , y^{(T+1)} - u \ra.
         \label{eq_lower_bd}
    \end{align}
    
    Next, we focus on the first term on the right hand side of \eqref{thm_eq2}. Using the definitions of the functions $\Omega_u(\cdot)$ and $H_u(\cdot)$ as given in   \eqref{def_Omega_u} and \eqref{def_H_u}, respectively, we have that
    \begin{align}
        H_u(x_0) &= \Omega(\nabla \psi (x_0)) - \la \nabla \psi (x_0), u \ra \nonumber\\
        & =  \max_{x \in X} [\ip{\nabla \psi (x_0)}{x} - \psi(x)] - \la \nabla \psi (x_0), u \ra . \label{thm_eq3}
    \end{align}
    Since the function $\psi(\cdot)$ is strongly convex (Assumption \ref{asum_phi}), a unique maximizer exists for the optimization problem on the right hand side of \eqref{thm_eq3}, which is denoted by $\hat x$. The optimality conditions for $\hat x \in X$ yield
    \begin{align*}
        \la \nabla \psi(x_0) - \nabla \psi(\hat{ x}) , x - \hat x \ra \leq 0 \qquad\forall x\in X.
    \end{align*}
    Note that $x_0 \in X$, i.e., it is feasible based on the updates of Algorithm~\ref{algo_Popov} and, hence, we see that the solution is $\hat x = x_0$,  as the preceding inequality is satisfied. 
    Therefore, by using this  in the right hand side of \eqref{thm_eq3}, we obtain
    \begin{align}
        H_u(x_0) = \la \nabla \psi (x_0), x_0 - u \ra - \psi(x_0) . \label{eq_hu_x1}
    \end{align}
    In the same way, we can derive the relation for the second term on the right hand side of \eqref{thm_eq2}, and obtain
    \begin{align}
        H_u(x_{T+1}) = \la \nabla \psi (x_{T+1}), x_{T+1} - u \ra - \psi(x_{T+1}) . \nonumber
    \end{align}
    Using the $\alpha$-strong convexity of $\psi(\cdot)$ (Assumption \ref{asum_phi}) in the preceding relation, we can further obtain
    \begin{align}
        H_u(x_{T+1}) \geq - \psi(u) + \frac{\alpha}{2} \|u - x_{T+1} \|^2 . \label{eq_hu_xT}
    \end{align}
    Substituting the relations obtained in \eqref{eq_lower_bd}, \eqref{eq_hu_x1}, and \eqref{eq_hu_xT} back in \eqref{thm_eq2}, we get
    \begin{align}
        &\gamma (T+1) \la F(u) , y^{(T+1)} - u \ra \leq \psi(u) - \psi(x_0) \nonumber\\
        & - \la \nabla \psi (x_0), u - x_0 \ra - \frac{\alpha}{2} \|u - x_{T+1} \|^2 + \sum_{t=0}^T \epsilon_t . \label{thm_eq4}
    \end{align}
    Now the quantity $\psi(u) - \psi(x_0) - \la \nabla \psi (x_0), u - x_0 \ra$ on the right hand side of \eqref{thm_eq4} is denoted by $B_{\psi}(u,x_0)$ which represents the Bregman divergence between the point $u \in X$ and $x_0 \in X$ induced by the function $\psi$. In addition, the last term on the right hand side of \eqref{thm_eq4} can be upper bounded using  Lemma~\ref{lem_eps_bound}. With these steps, \eqref{thm_eq4} reduces to
    \begin{align*}
        &\gamma (T+1) \la F(u) , y^{(T+1)} - u \ra + \frac{\alpha}{2} \|u - x_{T+1} \|^2 \nonumber\\
        &\leq B_{\psi}(u,x_0) + \frac{2\gamma^2 L^2}{\alpha} \|y_0 - x_0\|^2 . \nonumber
    \end{align*}
    In the preceding relation, we can ignore the second term on the left hand side and take the maximum with respect to $u \in X$ on both sides of the relation. Additionally, using the dual gap function (see~\eqref{dg_func}), we obtain
    \begin{align*}
        \gamma (T+1) G(y^{(T+1)}) \leq \max_{u\in X}B_{\psi}(u,x_0) + \frac{2\gamma^2 L^2}{\alpha} \|y_0 - x_0\|^2.
    \end{align*}
    Dividing by $\gamma (T+1)$ both sides of the preceding relation, we obtain
    \begin{align}
        G(y^{(T+1)}) \leq \frac{\max_{u \in X} B_{\psi}(u,x_0)}{(T+1) \gamma} + \frac{2\gamma L^2}{(T+1) \alpha} \|y_0 - x_0\|^2 . \nonumber
    \end{align}
    The result follows by using $\gamma = \frac{\alpha}{2L}$. Moreover, since the set $X$ is compact by Assumption \ref{asum_compact}, the quantity $\max_{u \in X} B_\psi(u,x_0)$ is finite.
\end{proof}

Theorem \ref{thm_popov_rates} shows $O(1/T)$ convergence rate of the dual gap function for the Popov mirror-prox which is the same as that of Korpelevich mirror-prox studied in \cite{nemirovski2004prox}, but with an additional term of $\frac{L}{T+1} \|y_0-x_0\|^2$ coming from the upper bound on $\sum_{t=0}^T \epsilon_t$ in Lemma \ref{lem_eps_bound}. We can eliminate the additional term by initializing Algorithm~\ref{algo_Popov} with $x_0 =y_0\in X$, hence giving a rate result similar to that of the Korpelevich method but with only a half of the mapping evaluations.

\section{Simulations}\label{sec_simulation}
In this section, we present simulation results for Popov and Korpelevich variants of the mirror-prox algorithm on a two player matrix game given by
\begin{align}
    \text{Player 1:} \quad &\min_{x_1\in \Delta_{2}} \;\; 
    \la x_1,A x_1\ra + \la x_1, B x_2\ra + \la p,x_1\ra \nonumber\\
    \text{Player 2:} \quad & \min_{x_2 \in \Delta_{2}} \;\; 
    \la x_2, C x_2\ra + \la x_1, D x_2\ra  + \la q, x_2\ra, \label{matrix_game}
\end{align}
%
where $\Delta_{2}$ is the probability simplex in $\R^{2}$, i.e., $\Delta_{2}=\{(h_1,h_2)\mid h \succcurlyeq 0, \ h_1+h_2 =1\}$. 
Hence, Assumption~\ref{asum_compact} is satisfied. 
The game can be modeled as a VI problem with 
the mapping $F(\cdot)$ obtained by differentiating each agent's loss function with respect to their respective decision variables, i.e.,
\begin{align}
  F(x) = \begin{bmatrix}
        A+A^T & B \\
        D^T & C+C^T
    \end{bmatrix} x + \begin{bmatrix}
        p \\
        q
    \end{bmatrix} , \label{mapping_F}
\end{align}
where $x = [x_1,x_2]^T \in \R^4$. Thus, the two player matrix game problem is equivalent to VI$(X,F)$ with $X=\Delta_2\times\Delta_2$ and the mapping $F(\cdot)$ as in~\eqref{mapping_F}.
The matrices $A$, $B$, $C$, and $D$ are generated to make the Jacobian of $F(\cdot)$, denoted by $\nabla F$, positive semidefinite with eigenvalues in the interval $[0,100]$. Hence, both Assumptions \ref{asum_monotone} and \ref{asum_lipschitz} are satisfied with $L=100$. 
The vectors $p$ and $s$ are sampled from a random normal distribution. For more details about the generation process see \cite[Section 7]{chakraborty2024random}.
We experiment with two different functions $\psi(\cdot)$: 
(i)~{\it Entropic case}: 
$\psi(x) = \sum_{j=1}^2 \sum_{i=1}^2 x_j^{(i)}(\ln x_j^{(i)}-1)$, which implies $\nabla \psi(x) = [\ln x_1^{(1)},\ln x_1^{(2)},\ln x_2^{(1)},\ln x_2^{(2)}]^T$, where $x_j = [x_j^{(1)}, x_j^{(2)}] \in \R^2$ for $j=1,2$. 
(ii)~{\it Euclidean case}: $\psi(x) = \frac{1}{2} \|x\|^2$, which implies $\nabla \psi(x) = x$. 
For the entropic case, $\nabla^2 \psi(x) \succcurlyeq \frac{I}{\max_{j,i} x_j^{(i)}} \succcurlyeq I$, since $\max_{j,i} x_j^{(i)} \leq 1$ due to the probability simplex constraints,
where $I$ is the identity matrix. For the Euclidean case, $\nabla^2 \psi(x) = I$. Thus, in both cases, Assumption~\ref{asum_phi} is satisfied with $\alpha = 1$. We run Algorithm~\ref{algo_Popov} using the step size $\gamma = \frac{\alpha}{2L} = \frac{1}{2L}$ and compare it with the Korpelevich mirror-prox algorithm~\cite{nemirovski2004prox} with the prescribed step size $\gamma = \frac{\alpha}{\sqrt{2}L}$. 
For the entropic case, the inverse transformation $\nabla \psi(\xi)^{-1} = \exp(\xi)$ preserves the nonnegativity of the variables. Therefore, the iterates $y_{t+1}$ and $x_{t+1}$ at step 3 of 
Algorithm~\ref{algo_Popov} are given by
\begin{align}
   &[y_{t+1}]_i = [z_{y,t+1}]_i [x_t]_i\exp(- \gamma [F(y_t))]_i), \quad i=1,2,3,4, \nonumber\\
   &[x_{t+1}]_i = [z_{x,t+1}]_i [x_t]_i\exp(- \gamma [F(y_{t+1})]_i), \quad i=1,2,3,4,\nonumber
\end{align}
where 
$[u]_i$ denotes the $i$th coordinate entry of a vector $u$, while $z_{y,t+1}$ and $z_{x,t+1}$ are normalization factors, i.e., 
for $i=1,2$ (Player~1),
\[[z_{y,t+1}]_i = \frac{1}{[x_t]_1\exp(- \gamma [F(y_t))]_1+[x_t]_2\exp(- \gamma [F(y_t))]_2},\]
and for $i=3,4$ (Player~2),
\[[z_{y,t+1}]_i = \frac{1}{[x_t]_3\exp(- \gamma [F(y_t))]_3+[x_t]_4\exp(- \gamma [F(y_t))]_4}.\]
The definition of $z_{x,t+1}$ is analogous (see~\cite{bubeck2015convex,chakraborty2022sparse}).

For the Euclidean case, the iterates $y_{t+1}$ and $x_{t+1}$ are: 
\begin{align*}
   &y_{t+1}=\Pi_{X}[x_t-\g F(y_t)], \cr
   &x_{t+1} = \Pi_{X} [x_t- \gamma F(y_{t+1})],
\end{align*}
where $\Pi_X[z]$ denotes the projection of a point $z$ on the set $X = \Delta_2 \times \Delta_2$ in the standard Euclidean norm. The projection for the Euclidean norm on the simplex is costly, as the nonnegativity of the iterates is not preserved directly from the updates, and we use the bisection method for simplex projection~\cite[Algorithm 3]{blondel2014large}. 

We initialize Algorithm~\ref{algo_Popov} with $x_0=y_0=[0.5,0.5]^2$. 
The dual gap function cannot be evaluated exactly and we estimate it empirically by generating $200000$ random samples within the simplex $\Delta_{2} \times \Delta_{2}$ and determining the maximum value of the inner product in \eqref{dg_func}. Note that this empirical approximation can be erroneous and introduce a variance between the empirical and the actual dual gap functions depending on the number of the samples generated. To lower the variance, we considered the matrix game of dimension~2.

Fig.~\ref{fig_mat_game} shows two different experiments with randomly generated Jacobian $\nabla F(\cdot)$. The left plot of Fig.~\ref{fig_mat_game} shows a superior performance for the entropic $\psi$ to that of the Euclidean $\psi$, whereas in the right plot of Fig.~\ref{fig_mat_game}, it is vice versa for both Korpelevich and Popov mirror-prox methods. Additionally, in both plots, the performance of Korpelevich and Popov mirror-prox methods are similar since we have chosen the same $x_0$ and $y_0$ for the later in order to eliminate the error term $\frac{L}{T}\|y_0-x_0\|^2$ showing up in the analysis. Moreover, Popov mirror-prox is computationally less expensive as compared to the Korpelevich mirror-prox.


\begin{figure}[t]\centering
	\begin{subfigure}{.49\linewidth}
		\includegraphics[width=1\linewidth, height = 0.75\linewidth]
		{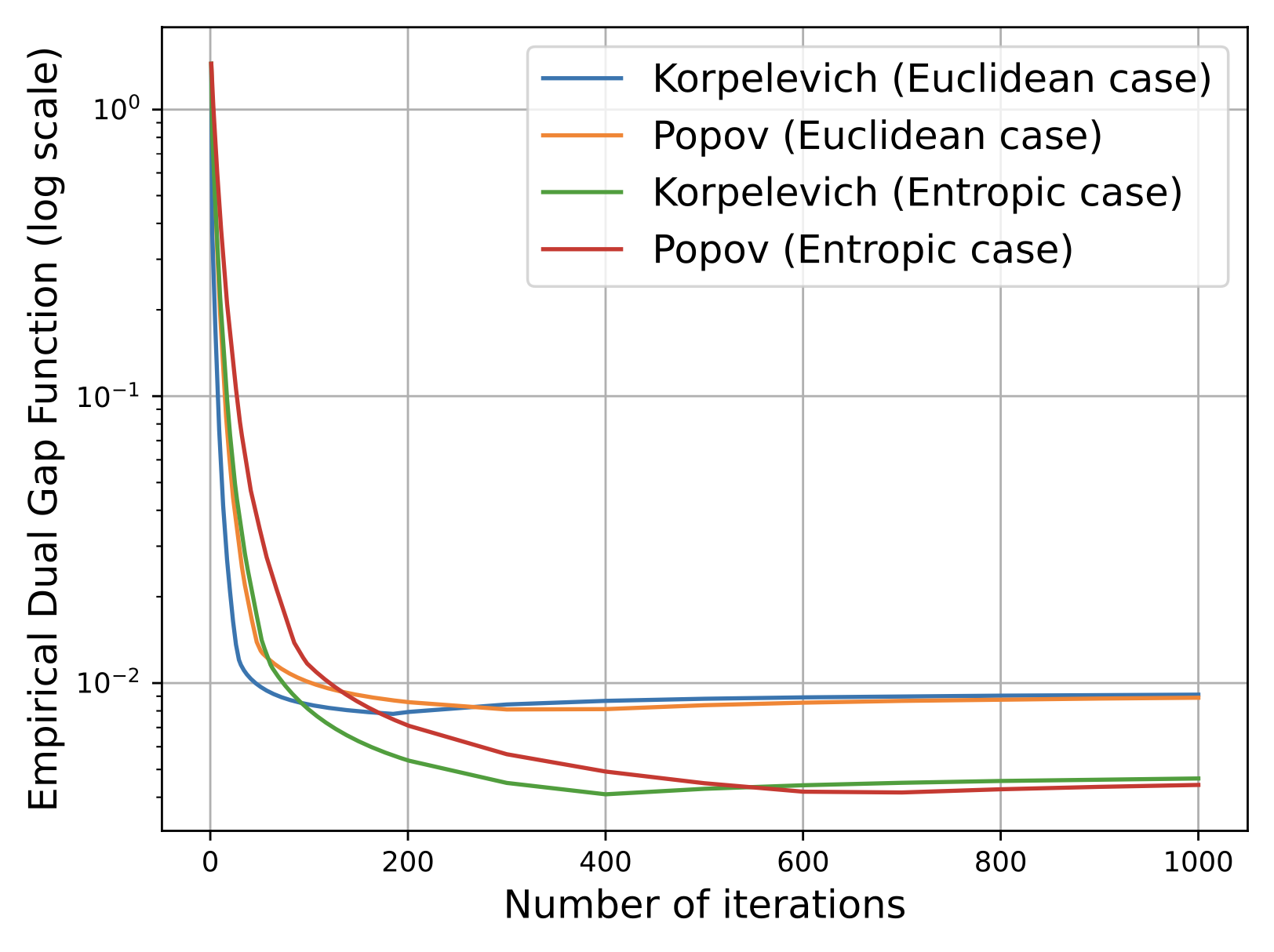}
	\end{subfigure}
	\begin{subfigure}{.49\linewidth}
		\includegraphics[width=1\linewidth,height = 0.75\linewidth]
		{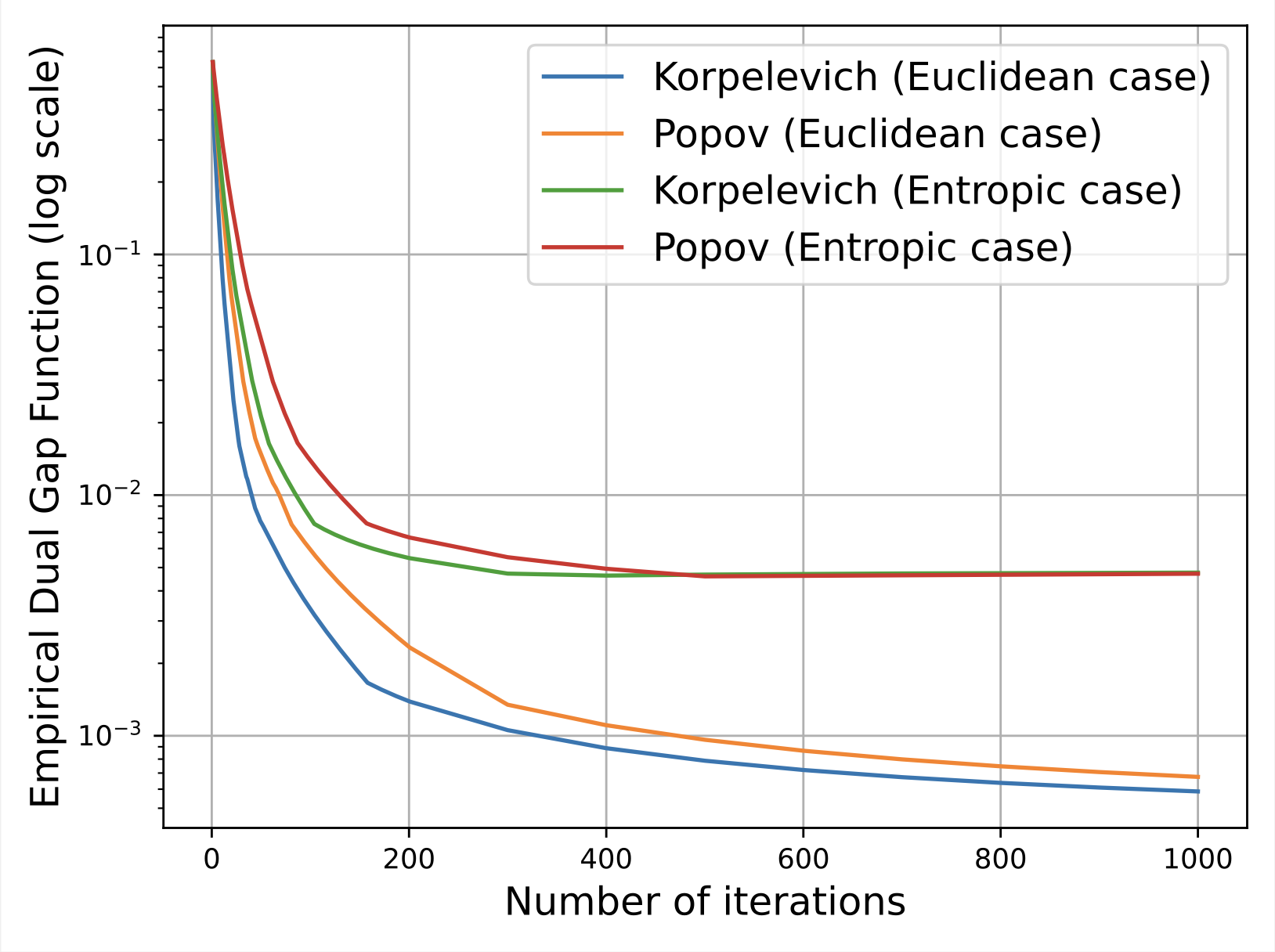}
	\end{subfigure}
 \caption{Two random experimental results for the matrix game}
\label{fig_mat_game}
\end{figure}



\section{Conclusions}\label{sec_conclusion}
In this paper, we discussed the conceptual mirror-prox algorithm and its implementation difficulties. Then, we developed an efficient implementation of the mirror-prox algorithm by using Popov method. We established the convergence rate of $O(1/T)$ for this method and, also, simulated its performance for a matrix game to support our results. As a future direction, one can extend the Popov method to stochastic mirror-prox. Also,  one may design randomized projection algorithms 
to alleviate the projection on the set $X$ when the set is complex and it is difficult to project on.






\bibliographystyle{IEEEtran}
\bibliography{IEEEabrv,reference}

\end{document}